\documentclass[11pt,psamsfonts]{amsart}
\usepackage{amsmath}
\usepackage{amsthm}
\usepackage{amssymb}
\usepackage{amscd}
\usepackage{amsfonts}
\usepackage{amsbsy}
\usepackage{setspace}
\usepackage{graphicx}
\usepackage{color}
\usepackage{calc}
\usepackage{subfigure}
\usepackage{caption}
\usepackage[font=small,labelfont=bf]{caption}
\DeclareCaptionFont{tiny}{\tiny}
\usepackage{graphicx}
\usepackage[utf8]{inputenc}
\usepackage[T1]{fontenc}
\usepackage{lmodern}
\usepackage{array} 
\usepackage{paralist} 
\usepackage{caption}
\usepackage{booktabs}
\usepackage{multirow}
\newcommand{\N}{\ensuremath{\mathbb{N}}}
\newcommand{\Q}{\ensuremath{\mathbb{Q}}}

\newcommand{\SSS}{\ensuremath{\mathbb{S}}}

\newcommand{\CPn}{\ensuremath{\mathbb{C}\text{\textbf{P}}^n}}
\newcommand{\HPn}{\ensuremath{\mathbb{H}\text{\textbf{P}}^n}}

\newcommand{\mbS}{\mathbb{S}}
\newcommand{\MZ}{\mathcal{Z}}

\newtheorem{theorem}{Theorem}

\theoremstyle{definition}
\newtheorem{remark}[theorem]{Remark}

\newtheorem{definition}[theorem]{Definition}

\begin{document}
        
\title[Periods of Morse--Smale diffeomorphisms ]
{Periods of Morse--Smale diffeomorphisms\\ on $\SSS^n$, $\SSS^m \times \SSS^n$, $\CPn$ and $\HPn$.}

\author[C. Cuf\'i-Cabr\'e and J. Llibre]
{Clara Cuf\'i-Cabr\'e and Jaume Llibre}

\address{Departament de Matem\`{a}tiques, Universitat Aut\`{o}noma de Barcelona, 08193 Bellaterra, Barcelona, Catalonia, Spain}
\email{clara@mat.uab.cat, jllibre@mat.uab.cat}


\date{}
\dedicatory{}

\maketitle

\begin{abstract}
We study the set of periods of the Morse--Smale diffeomorphisms on the $n$-dimensional sphere $\SSS^n$, on products of two spheres of arbitrary dimension $\SSS^m \times \SSS^n$ with $m \neq n$,
on the $n$-dimensional complex projective space $\CPn$ and on the  $n$-dimensional quaternion projective space $\HPn$. We classify the minimal sets of Lefschetz periods for such Morse--Smale diffeomorphisms. This characterization is done using the induced maps on the homology. The main tool used is the Lefschetz zeta function.
\end{abstract}

\section{Introduction}

Understanding the periodic orbits and the set of periods of a map is a very important problem in dynamical systems. The Lefschetz numbers are one of the most useful tools to study the existence of fixed points and periodic orbits of self-maps on compact manifolds. In this paper we obtain information on the set of periods of certain diffeomorphisms on compact manifolds using the Lefschetz zeta function, which is a generating function of the Lefschetz numbers of the iterates of a map.

Let $M$ be a compact manifold, let $f: M \to M$ be a continuous map, and denote by $f^m$  the $m-$th iterate of $f$. A point $x \in M$ such that $f(x)= x$ is called a \emph{fixed point}, or a {\it periodic point of period} $1$ of $f$. A point $x \in M$ is called \emph{periodic of period $k>1$} if $f^k(x)=x$ and $f^m(x)\neq x$ for all $m=1,\ldots,k-1$, and the set formed by the iterates of $x$, i.e. $\{x, \, f(x), \, \dots, \, f^{k-1}(x) \}$, is called the \emph{periodic orbit} of the periodic point $x$.

As usual $\N$ denotes the set of all positive integers. Then $\text{Per}(f)$ is the set $\{k \in \N \, : \, f\, \, \text{has a periodic orbit of period} \, k \}.$

A fixed point $x$ of a $C^1$ map $f$ is called \emph{hyperbolic} if all the eigenvalues of $Df(x)$ have modulus different than one. A periodic point $x$ of $f$ of  period $k$ is called a \emph{hyperbolic periodic point}  if it is a hyperbolic fixed point of $f^k$.

We denote by $\text{Diff}(M)$ the space of all $C^1$ diffeomorphisms on a compact Riemannian manifold $M$. 
We say that two maps $f, \, g \in \text{Diff}(M)$ are \emph{topologically equivalent} if there exists a homeomorphism $h : M \to M$ such that $h \circ f = g \circ h$. A map $f \in \text{Diff}(M)$ is called \emph{structurally stable} if there exists a neighbourhood $U \subset \text{Diff}(M)$ of $f$ such that every $g\in U$ is topologically equivalent to $f$.

A map $f \in \text{Diff}(M)$ is \emph{quasi unipotent} on homology if all the eigenvalues of the nontrivial induced maps on the homology groups of $M$ with rational coefficients, are roots of unity. We denote by $\mathcal{F}(M)$ the set of elements of $\text{Diff}(M)$ being quasi unipotent on homology and having finitely many periodic points, all of them hyperbolic. Among the maps in $\mathcal{F}(M)$ there are the \emph{Morse-Smale diffeomorphisms} (see \cite{fs,shub}). Those are very important from the dynamical point of view because the family of the Morse--Smale diffeomorphisms is structurally stable inside the class of all diffeomorphisms (see for details \cite{fs, PS, robinson, smale}).

Even if we will refer to Morse--Smale diffeomorphisms along the paper because of their dynamical interest, all the results hold for any map of the class $\mathcal{F} (M)$.
In order to define the Morse-Smale diffeomorphisms we introduce the following concepts. 

We say that $x \in M$ is a \emph{nonwandering point} of $f$ if for any neighborhood $U$ of $x$ there is a positive integer $m$ such that $f^m(U)\cap U \neq \emptyset$. We denote by $\Omega(f)$ the set of nonwandering points of $f$. Clearly if $\gamma$ is a periodic orbit of $f$, then $\gamma \subseteq \Omega(f)$.

Let $d$ be the metric on $M$ and suppose that $x \in M$ is a hyperbolic fixed point of $f$. We define the \emph{stable manifold} of $x$ as the set
$$
W^s(x) = \{ y \in M \, : \, d(x, \, f^m(y)) \to 0 \text{ as } m \to \infty \},
$$
and the unstable manifold of $x$ as the set
$$
W^u(x) = \{ y \in M \, : \, d(x, \, f^{-m}(y)) \to 0 \text{ as } m \to \infty \}.
$$

In the same way we define the stable and unstable manifolds of a hyperbolic periodic point $x \in M$ of period $k$ as the stable and unstable manifolds of the hyperbolic fixed point $x$ under $f^k$, respectively. We say that the submanifolds $W^s(x)$ and $W^u(x)$ have a \emph{transversal intersection} if at every point of intersection, their separate tangent spaces at that point together generate the tangent space of the ambient manifold $M$ at that point.

A diffeomorphism $f: M \to M$ is \emph{Morse--Smale} if
\begin{itemize}
\item[$(1)$] $\Omega(f)$ is finite,
\item[$(2)$] all the periodic points of $f$ are hyperbolic, and
\item[$(3)$] for every $x, \, y \in \Omega(f)$, $W^s(x)$ and $W^u(y)$ have a transversal intersection.
\end{itemize}

It is known that condition $(1)$ implies that $\Omega(f)$ is the set of all periodic points of $f$, see again  \cite{fs, robinson, smale}.

In the last quarter of the 20th century there appeared some papers dedicated to understand the connections between the dynamics of the Morse--Smale diffeomorphisms and the topology of the manifold where they are defined. Without trying to be exhaustive, see for instance \cite{FN, Na, shub, SS, smale}. This interest continues during this first part of the 21st century, see for example \cite{BGS, BGP, GLM, LS1, LS2, LS3}.

In this paper we study the set of periods of Morse--Smale diffeomorphisms on the $n-$dimensional sphere $\SSS^n$, on products of two spheres of arbitrary dimension $\SSS^m \times \SSS^n$ with $m \neq n$, and on the projective spaces $\CPn$ and $\HPn$. More precisely, our goal is to describe the set $\text{MPer}_L(f)$ (see Definition \ref{def_MPerLf}) that those diffeomorphisms  can exhibit for arbitrary values of $n$ and $m$.

The set of periods for Morse--Smale diffeomorphisms on a product of any number of spheres of the same dimension has been studied in  \cite{BGS}. For the particular case of $\SSS^n$, we give more detailed results considering the orientation of the diffeomorphisms. The set of periods for Morse--Smale diffeomorphims on the two-dimensional sphere has been studied with more details in \cite{BHN, GL1}. 

In \cite{graff-justyna} the authors study the Lefschetz periodic point free self-maps on $\SSS^m \times \SSS^n$, $\CPn$ and $\HPn$. Our results give extended information in the same line for the Morse--Smale diffeomorphisms.

The main results of this paper are Theorems \ref{th_sn}, \ref{th_mn_p} and \ref{th_cpn_hpn}, where we characterize the possible sets for $\text{MPer}_L(f)$ in function of the action of $f$ on the homology and on the parity of the numbers $n$ and $m$.

Related with these results the reader can look at the set of periods for homeomorphisms (respectively continuous maps) on $\SSS^n$ and on $\SSS^m \times \SSS^n$ which have been studied in \cite{GL3} (respectively \cite{GL4}).

We also remark that the results obtained along this paper hold 
in any compact manifold with the same homology as the manifolds considered here. More precisely, they hold for any manifold homotopy equivalent to $\SSS^n$, $\SSS^m \times \SSS^n$, $\CPn$ and $\HPn$, respectively.

\section{Lefschetz numbers and the Lefschetz zeta function} \label{s2}

Let $f: M \to M$ be a continuous map on a compact manifold of dimension $n$. We denote by  $H_0(M, \, \Q), \, \dots, \, H_n(M, \, \Q)$ the homology groups of $M$ with rational coefficients. A continuous map $f: M \to M$ induces $n+1$ morphisms on the homology groups of $M$,
$
f_{*i} : H_i(M, \, \Q) \to H_i(M, \, \Q), \, i \in \{0,\,  \dots, \, n \}.
$
For more details, see \cite{vick}.

\begin{definition}
Let $M$ be a compact manifold and let $f: M\to M$ be a continuous map. The \emph{Lefschetz number of $f$} is defined as
$$
L(f) = \sum_{i=0}^{n} \,  (-1)^i \, \text{tr}(f_{*i}),
$$
where tr$(f_{*i})$ denotes the trace of $f_{*i}$.
\end{definition}

A very important result which relates the Lefschetz number of a map $f$ and the existence of fixed points of $f$ is the Lefschetz Fixed Point Theorem, which says that if $L(f) \neq 0$, then $f$ has a fixed point. For a proof, see \cite{brown}.

In order to obtain information about the set of periods of a map $f$ it is useful to have information of the sequence $\{L(f^m)\}_{m=0}^\infty$ of the Lefschetz numbers of all the iterates of $f$. 

\begin{definition}\label{zeta_function}
Let $f: M \to M$ be a continuous map on a compact manifold. The Lefschetz zeta function of $f$ is defined as
$$
\mathcal{Z}_f(t) = \exp \bigg{(} \sum_{m=1}^\infty \frac{L(f^m)}{m} \, t^m \bigg{)}.
$$
\end{definition}

This function generates the whole sequence of Lefschetz numbers of $f$, and it may be independently computed through
\begin{equation} \label{zeta_product}
\mathcal{Z}_f(t) = \prod_{k=0}^n \det \, (I_{n_k} - t\, f_{*k})^{(-1)^{k+1}},
\end{equation}
(see for more details \cite{franks}) where $n$ is the dimension of $M$, $n_k$ is the dimension of  $H_k(M, \, \mathbb{Q})$, $I_{n_k}$ is the $n_k \times n_k$ identity matrix, and we take $\det \, (I_{n_k} - t \, f_{*k}) = 1$ if $n_k = 0$. Note that the expression given in \eqref{zeta_product} is a rational function of $t$. 

For a $C^1$ map $f$ having a finite number of periodic points, all of them being hyperbolic, we give another characterization of the Lefschetz zeta function introduced by Franks in \cite{franks}.

Let $f: M \to M$ be a $C^1$ map on a compact manifold without boundary, and let $\gamma$ be a hyperbolic periodic orbit of $f$ of period $p$. For each $x \in \gamma$, let $E_x^u$ denote the linear subspace of the tangent space $T_xM$ of $M$ at $x$, generated by the eigenvectors of $Df^p$ corresponding to the eigenvalues whose moduli are greater than $1$, and let $E_x^s$ denote the linear subspace of $T_xM$ generated by the remaining eigenvectors. We denote by $u$ and $s$ the dimensions of the spaces $E^u_x$ and $E^s_x$, respectively.

We define the \emph{orientation type} $\Delta$ of $\gamma$ to be $+1$ if $Df^p(x): E_x^u \to E_x^u$ preserves orientation, that is, if $\det Df^p(x) >0$ with $x\in \gamma$, and to be $-1$ if it reverses orientation, that is, if $\det Df^p(x) <0$. Note that the definitions of $\Delta$ and $u$ do not depend on the periodic point $x$, but only on the periodic orbit $\gamma$.

For a $C^1$ map $f: M \to M$ having only finitely many periodic orbits, all of them hyperbolic, we define the \emph{periodic data}, $\Sigma$, as the collection composed by all triples $(p, \, u, \, \Delta)$ corresponding to all the hyperbolic periodic orbits of $f$, where a same triple can occur more than once provided that it corresponds to different periodic orbits. The following result was proved by Franks in \cite{franks}.

\begin{theorem} \label{th_franks}
Let $f: M \to M$ be a $C^1$ map on a compact manifold without boundary, having finitely many periodic points, all hyperbolic, and let $\Sigma$ be the periodic data of $f$. Then the Lefschetz zeta function of $f$ satisfies
\begin{equation*} \label{zeta_data}
\mathcal{Z}_f(t) = \prod_{(p, \, u, \, \Delta) \in \Sigma} \, (1-\Delta\, t^p)^{(-1)^{u+1}}.
\end{equation*}
\end{theorem}

Note that the Morse--Smale diffeomorphisms satisfy the hypotheses of Theorem  \ref{th_franks}. Then the result  will be useful to obtain information on the set of periods for Morse-Smale diffeomorphisms from the comparsion of the expressions of the Lefschetz zeta functions in Theorem \ref{th_franks} and in \eqref{zeta_product}.

\section{Minimal set of Lefschetz periods of Morse--Smale diffeomorphisms}\label{s3}

\begin{definition} \label{def_MPerLf}
Let $f$ be a map satisfying the hypotheses of Theorem $\ref{th_franks}$. The \emph{minimal set of Lefschetz periods of $f$}, $\text{MPer}_L(f)$, is the set given by the intersection of all sets of periods forced by the different representations of $\mathcal{Z}_f(t)$ as products of the form $(1 \pm t^p)^{\pm1}$.
\end{definition}

As an example consider the following Lefschetz zeta function of a Morse--Smale diffeomorphism $f$ on the four-dimensional torus $\mathbb{T}^4$,
\begin{align}
\begin{split} \label{zeta_tor_exemple}
\mathcal{Z}_f(t) & = \frac{(1-t^3)^2(1+t^3)}{(1-t)^6(1+t)^3} 
= \frac{(1-t^3)(1-t^6)}{(1-t)^6(1+t)^3} = \frac{(1-t^3)(1-t^6)}{(1-t)^3(1-t^2)^3}  \\
&= \frac{(1-t^3)^2(1+t^3)}{(1-t)^3(1-t^2)^3},
 \end{split}
\end{align}
see for instance \cite{GL2}, which can be expressed in four ways as a product of factors of the form $(1 \pm t^p)^{\pm1}$ as a quotient of polynomials of degree $9$.

By Theorem \ref{th_franks}, the first expression of the Lefschetz zeta function \eqref{zeta_tor_exemple} ensures the existence of periodic orbits of periods $1$ and $3$ for $f$. In the same way, the second expression of $\mathcal{Z}_f(t)$ provides the periods $\{1, \, 3, \, 6\}$ for $f$, the third expression of $\mathcal{Z}_f(t)$  provides the periods $\{1, \, 2, \, 3, \, 6\}$, and finally the fourth expression provides the periods $\{1, \, 2, \, 3\}$. In this case we have that the minimal set of Lefschetz periods of $f$ is
$$
\text{MPer}_L(f) = \{1, \, 3\} \cap \{1, \, 3, \, 6\} \cap \{1,\, 2, \, 3, \, 6\} \cap \{1, \, 2, \, 3\} = \{1, \, 3\}.
$$
Note that if $\mathcal{Z}_f(t)$ is constant equal to $1$, then $\text{MPer}_L(f) = \emptyset$.

\begin{remark} \label{tres_maneres}
Even if the minimal set of Lefschetz periods of a Morse--Smale diffeomorphism is empty, one can still obtain some information on the set of periods from Theorem \ref{th_franks}. For example, suppose that the Lefschetz zeta function of a map $f$ satisfying the hypotheses of Theorem $\ref{th_franks}$ is $\mathcal{Z}_f(t) = 1+t^2$. It can be expressed as products of terms of the form $(1\pm t^p)^{\pm1}$ in infinitely many ways,
\begin{equation*}
1+t^2 = \frac{1-t^4}{1-t^2} = \frac{1-t^4}{(1+t)(1-t)} = \frac{1-t^8}{(1+t)(1-t)(1+t^4)},
\end{equation*}
and so on. In this case we have clearly  $\text{MPer}_L(f) = \emptyset$, but each of the infinitely many expressions of $\MZ_f(t)$ forces either the period $2$ or the periods $\{2, \,  4\}$, or the periods $\{1, \,  4\}$. Then, $f$ has either a periodic orbit of period $2$, or  periodic orbits of periods $2$ and $4$, or periodic orbits of periods $1$ and $4$.
\end{remark}

\begin{remark} \label{remark_nparell}
By Theorem $\ref{th_franks}$ an even period $n$, can never be contained in the set $\text{MPer}_L(f)$. Indeed, the following expressions
$$
1-t^n = (1+t^{n/2})(1-t^{n/2}), \quad 1+t^n = \frac{1-t^{2n}}{1-t^n} = \frac{1-t^{2n}}{(1+t^{n/2})(1-t^{n/2})},
$$
show that if the term $1-t^n$ or $1+t^n$, with $n$ even, appears in one of the expressions of $\mathcal{Z}_f(t)$, one can always obtain a new expression of $\mathcal{Z}_f(t)$ where the period $n$ does not appear. 
\end{remark}

\begin{remark} \label{remark_una_expr}
Along the paper, for every possible Lefschetz zeta function of a given map $f$, in general we will write only one of the possible equivalent expressions of $\mathcal{Z}_f(t)$. We will provide the expression of $\mathcal{Z}_f(t)$ that forces a smaller set of periods, and consequently it will be sufficient to describe the set $\text{MPer}_L(f)$.
From remark  \ref{remark_nparell} we note that providing an expression of $\mathcal{Z}_f(f)$ that forces only even periods is sufficient to ensure that $\text{MPer}_L(f) = \emptyset$.
\end{remark}

Finally note that  $\text{MPer}_L(f)$ is contained in the set of periods that are conserved under homotopy. Indeed, for a Morse--Smale diffeomorphism $f: M \to M$ on a compact manifold consider the set
$
\text{MPer}_{ms}(f) :=  \bigcap_{h\sim f} \, \text{Per}(h),
$
where $h$ runs over all the Morse--Smale diffeomorphisms of $M$ which are homotopic to $f$. 
Then it is clear that
$$
\text{MPer}_L(f) \subseteq  \text{MPer}_{ms}(f),
$$
because two homotopic maps on a manifold $M$ induce  the same Lefschetz zeta functions.

\section{Periods of Morse--Smale diffeomorphisms on $\SSS^n$}\label{s4}

Let $n \in \N$ and let $\SSS^n$ be the $n-$dimensional sphere. 
The homology groups of $\SSS^n$ over $\Q$ are
\begin{equation*} \label{homologia_sn}
H_k(\mathbb{S}^n, \, \Q) =
\begin{cases}
\Q \quad \text{if} \; \, k  \in \{0, \, n\}, \\
0 \quad \text{otherwise}.
\end{cases}
\end{equation*}

For a continuous map $f: \SSS^n \to \SSS^n$ the nontrivial induced maps on the homology can be written as the integer matrices $f_{*0} = (1)$ and $f_{*n} = (d)$, where $d$ is called the \emph{degree} of $f$.

Let $f: \SSS^n \to \SSS^n$ be a Morse--Smale diffeomorphism.
As we already mentioned before, the linear maps induced on the homology are quasi unipotent, which means that all their eigenvalues are roots of unity. Then we must have either $d=1$, or $d=-1$. Also, the orientation of $f$ is constant on $\SSS^n$ and is determined by the sign of $d$.

\begin{theorem} \label{th_sn}
Let $f: \mbS^n \to \mbS^n$ be a Morse--Smale diffeomorphism. Then, 
\begin{enumerate}[(a)]
\item If $n$ is even and $f$ preserves the orientation, then $\text{MPer}_L(f) = \{1\}$.
\item If $n$ is even and $f$ reverses the orientation, then $\text{MPer}_L(f) = \emptyset$ but $\{1, \, 2\} \cap \text{Per}(f) \neq \emptyset$. 
\item If $n$ is odd and $f$ preserves the orientation, then $\text{MPer}_L(f) =\emptyset$.
\item If $n$ is odd and $f$ reverses the orientation, then $\text{MPer}_L(f) =\{1\}$.
\end{enumerate}
\end{theorem}

\begin{proof}
Computing the Lefschetz zeta function for $f$ using equation \eqref{zeta_product}, we get
$$
\mathcal{Z}_f(t) = \prod_{k=0}^n \det \, (I_{n_k} - t\, f_{*k})^{(-1)^{k+1}}= (1-t)^{-1} \, (1-td)^{(-1)^{n+1}},
$$
and so for $n$ even we have
\begin{align*}
\mathcal{Z}_f(t)  = \frac{1}{(1-t)^2}  \, \text{ if } \,  d=1, \qquad
\mathcal{Z}_f(t)   = \frac{1}{(1-t)(1+t)} =  \frac{1}{1-t^2} \,  \text{ if } \, d=-1,
\end{align*}
and for $n$ odd we have
\begin{align*}
\mathcal{Z}_f(t)  \equiv 1 \, \text{ if } \, d=1, \qquad \mathcal{Z}_f(t)  = \frac{1+t}{1-t} \, \text{ if } \, d=-1.
\end{align*}
        
It is clear that $f$ satisfies the hypotheses of Theorem \ref{th_franks}. 
Then the statements follow directly applying Theorem \ref{th_franks} to each of the expressions obtained for the Lefschetz zeta function $\mathcal{Z}_f(t)$.
         
Indeed, for $1/(1-t)^2$ and $(1+t)/(1-t)$ we have $\text{MPer}_L(f) = \{1\}$, because any other expression of the same Lefschetz zeta function as a product of terms the form $(1\pm t^p)^{\pm 1}$ would provide at least the period $1$.
        
For the function $1/(1-t^2) = 1/((1-t)(1+t))$ it is clear that $\text{MPer}_L(f) = \emptyset$,  but by Remarks \ref{tres_maneres} and \ref{remark_nparell}, we can ensure that $\{1, \, 2\} \cap \text{Per}(f) \neq \emptyset$.  
Finally  we have $\text{MPer}_L(f) = \emptyset$ when $\mathcal{Z}_f(t) \equiv 1$.
\end{proof}

\section{Periods of Morse--Smale diffeomorphisms on $\SSS^m \times \SSS^n$}\label{s5}


Let $m, n \in \N$, $m \neq n$, and consider the product of spheres $\mathbb{S}^m \times \mathbb{S}^n$, with $m \neq n$. Applying Künneth's formula, the homology groups over $\Q$ of $\mathbb{S}^m \times \mathbb{S}^n$ can be easily computed and are given by
\begin{equation*}
H_k(\mathbb{S}^m \times \mathbb{S}^n, \, \Q) =
\begin{cases}
\Q \quad \text{if} \; \, k \in \{ 0,\, m,\, n,\, m+n \}, \\
0 \quad \text{otherwise}.
\end{cases}
\end{equation*}

Then for a continuous map $f: \SSS^m \times \SSS^n \to \SSS^m \times \SSS^n$ the linear maps induced on the homology are given as the integer matrices $f_{*0} = (1)$, $f_{*m} = (a)$, $f_{*n} = (b)$, $f_{*m+n} = (d)$, where $d$ is the degree of $f$. The rest of the induced maps are the zero map.

Let $f: \SSS^m \times \SSS^n \to \SSS^m \times \SSS^n$ be a Morse--Smale diffeomorphism. Since the linear maps induced by $f$ on the homology are quasi unipotent, in this case we have $a, \, b, \, d \in\{-1, \, 1\}$. Moreover, by the structure of the cohomology ring of $\SSS^m \times \SSS^n$, one has always that $ab=d$ (see \cite{graff-justyna}) and hence the possibilities for these values are restricted to $a=b =c=1$, $\{a, b\} = \{-1, 1\}$ and $d=-1$, and $a = b = -1, \, d=1$.

\begin{theorem} \label{th_mn_p}
Let $M= \SSS^m \times \SSS^n$ be with $m \neq n$, and let $f: M \to M$ be a Morse--Smale diffeomorphism.
\begin{enumerate}[(i)]
\item If $m$ and $n$ are even, then
\begin{enumerate}[(a)]
\item If $a=b=c=1$, then $\text{MPer}_L(f) = \{1\}$.
\item Otherwise, $\text{MPer}_L(f) = \emptyset$ but $\{1, \, 2\} \cap \text{Per}(f) \neq \emptyset$.
\end{enumerate}
\item If $m$ and $n$ are odd, then
\begin{enumerate}[(a)]
\item If $a=b=-1$ and $d=1$,  then $\text{MPer}_L(f) = \{1\} $.
\item Otherwise, $\text{MPer}_L(f) = \emptyset$.
\end{enumerate}
\item If $m$ is even and $n$ is odd, then
\begin{enumerate}[(a)]
\item If $b=d=-1$ and $a=1$, then $\text{MPer}_L(f) =  \{1 \}$.
\item Otherwise, $\text{MPer}_L(f) = \emptyset$.
\end{enumerate}
\item If $m$ is odd and $n$ is even, then
\begin{enumerate}[(a)]
\item If $a=d=-1$ and $b=1$, then $\text{MPer}_L(f) = \{1\}$.
\item Otherwise, $\text{MPer}_L(f) = \emptyset$.
\end{enumerate}
\end{enumerate}
\end{theorem}

\begin{proof}
Computing the Lefschetz zeta function for $f$ using equation \eqref{zeta_product} we obtain
\begin{equation*} \label{zeta_smxsn}
\begin{array}{rl}
\mathcal{Z}_f(t) =& \displaystyle\prod_{k=0}^{m+n} \det \, (I_{n_k} - t\, f_{*k})^{(-1)^{k+1}}\vspace{0,2cm}\\
=& (1-t)^{-1}(1-a\, t)^{(-1)^{m+1}} (1-b\, t)^{(-1)^{n+1}} (1-d \,  t)^{(-1)^{m+n+1}}.
\end{array}
\end{equation*}

Depending on whether $m$ and $n$ are even or odd, and for each allowed value of $a, \, b, \, d \in \{-1, \, 1\}$ with $ab=d$, the expressions obtained for the Lefschetz zeta functions in each case are displayed in Tables \ref{taula_mn_p}, \ref{taula_mn_s}, \ref{taula_m_p_n_s} and \ref{taula_m_s_n_p}.

The proof follows directly from the information obtained from the mentioned tables and applying Theorem \ref{th_franks}, taking into account the considerations of Remarks \ref{tres_maneres}, \ref{remark_nparell} and \ref{remark_una_expr}.
\end{proof}

\begin{table}[h]  \centering
\begin{tabular}{c  @{\hspace{1.3cm}} c}
\toprule[1.2pt]
Values for $a, \, b, \, d$ & $\MZ_{f}(t)$ \\ \midrule
$a=b=d=1$ & $\frac{1}{(1-t)^4}$ \\
$\{a, b, d\} = \{-1, 1\}$ with $ab=d$ & $\frac{1}{(1-t)^2} \frac{1}{(1+t)^2} = \frac{1}{(1-t^2)^2}$ \\ \bottomrule[1.2pt]
\end{tabular}
\caption{$\MZ_f(t)$ for a quasi-unipotent diffeomorphism $f$ on $\SSS^m \times \SSS^n$, with $m \neq n$ and $m, \, n$ even.}
\label{taula_mn_p}
\end{table}

\begin{table}[h]  \centering
\begin{tabular}{c @{\hspace{1.3cm}} c}
\toprule[1.2pt]
Values for $a, \, b, \, d$  & $\MZ_{f}(t)$ \\ \midrule
$a=b=d=1$ & $1$ \\
$a=b=-1, \, d=1$ &   $\frac{(1+t)^2}{(1-t)^2}$ \\
$\{a, \, b\} = \{-1, \, 1\}, \, d=-1$ & $1$ \\ \bottomrule[1.2pt]
\end{tabular}
\caption{$\MZ_f(t)$ for a quasi-unipotent diffeomorphism $f$ on $\SSS^m \times \SSS^n$, with $m \neq n$ and $m, \, n$ odd.}
\label{taula_mn_s}
\end{table}

\begin{table}[h]  \centering
\begin{tabular}{c @{\hspace{1.3cm}} c}
\toprule[1.2pt]
Values for $a, \, b, \, d$ & $\MZ_{f}(t)$ \\ \midrule
$a=b=d=1$ &  $1$ \\
$b=d=-1, \, a=1$ &  $\frac{(1+t)^2}{(1-t)^2}$ \\
$\{b, \, d\} = \{-1, \, 1\}, \, a=-1$ & $1$ \\ \bottomrule[1.2pt]
\end{tabular}
\caption{$\MZ_f(t)$ for a quasi-unipotent diffeomorphism $f$ on $\SSS^m \times \SSS^n$, with $m \neq n$ and $m$ even, $n$ odd.}
\label{taula_m_p_n_s}
\end{table}

\begin{table}[h]  \centering
\begin{tabular}{c @{\hspace{1.3cm}}  c}
\toprule[1.2pt]
Values for $a, \, b, \, d$ & $\MZ_{f}(t)$ \\ \midrule
$a=b=d=1$ & $1$ \\
$a=d=-1, \, b=1$ & $\frac{(1+t)^2}{(1-t)^2}$ \\
$\{a, \, d\} = \{-1, \, 1\}, \, b=-1$ & $1$ \\ \bottomrule[1.2pt]
\end{tabular}
\caption{$\MZ_f(t)$ for a quasi-unipotent diffeomorphism on $\SSS^m \times \SSS^n$, with $m \neq n$ and $m$ odd, $n$ even.}
\label{taula_m_s_n_p}
\end{table}

\section{Periods of Morse--Smale diffeomorphisms on $\CPn$ and $\HPn$}\label{s7}

Let $n \in \N$, $n \geq 1$ and let $\CPn$ be the $n$-dimensional complex projective space and let $\HPn$ be the  $n$-dimensional quaternion projective space. Along this section we consider these manifolds as real manifolds of dimension $2n$ and $4n$, respectively.

The homology groups of $\CPn$ over $\Q$ can be easily computed applying Künneth's formula and are given by
\begin{equation*} \label{homologia_sn}
    H_k(\CPn, \, \Q) = 
    \begin{cases}
    \Q \quad \text{if} \; \, k  \in \{0, \, 2, \, \dots, \,  2n\}, \\
    0 \quad \text{otherwise}.
    \end{cases}
\end{equation*}

For a continuous map $f: \CPn \to \CPn$, the  induced linear maps on the homology can be written as as integer matrices as 
$ f_{*k} = (d^{k/2}) $ for $k \in \{0, \, 2, \, 4, \, \dots , \, 2n\}$, with $ d \in \mathbb{Z}$, and $f_{*k} = (0)$ otherwise (see \cite{vick}).

Similarly the homology groups of $\HPn$ over $\Q$ are given by
\begin{equation*} 
H_k(\HPn, \, \Q) = 
\begin{cases}
\Q \quad \text{if} \; \, k  \in \{0, \, 4, \, \dots, \,  4n\}, \\
0 \quad \text{otherwise},
\end{cases}
\end{equation*}
and  for a continuous map $f: \HPn \to \HPn$, the  induced linear maps on the homology can be written  as 
$ f_{*k} = (d^{k/4}) $ for $k \in \{0, \, 4, \, 8, \, \dots , \, 4n\}$, with $ d \in \mathbb{Z}$, and $f_{*k} = (0)$ otherwise (see \cite{vick}).

Let $f: \CPn \to \CPn$ (respectively, $f: \HPn \to \HPn$) be a Morse-Smale diffeomorphism. 
Since the linear maps induced on the homology are quasi unipotent, we will have either $d=1$ or $d=-1$, which determines also whether $f$ preserves or reverses the orientation. 

\begin{theorem} \label{th_cpn_hpn}
Let $f: \mathbb{C}\text{\emph{\textbf{P}}}^n \to \mathbb{C}\text{\emph{\textbf{P}}}^n$ (respectively, $f: \mathbb{H}\text{\emph{\textbf{P}}}^n \to \mathbb{H}\text{\emph{\textbf{P}}}^n$)  be a Morse--Smale diffeomorphism. Then, 
\begin{enumerate}[(a)]
\item If $n$ is odd and $f$ reverses the orientation, then $\text{MPer}_L(f) = \emptyset$, but $\{1, \, 2\} \cap \text{Per}(f) \neq \emptyset$.
\item Otherwise, $\text{MPer}_L(f) =1$.
\end{enumerate}
\end{theorem}

\begin{proof}
We develop the proof for $\CPn$, being the proof for $\HPn$ completely analoguous.
	
We start with the case $d=1$. Computing the Lefschetz zeta function for $f$ using \eqref{zeta_product}, we get
$$
\mathcal{Z}_f(t) = \prod_{k=0}^{2n} \det \, (I_{n_k} - t\, f_{*k})^{(-1)^{k+1}}= \frac{1}{(1-t)^{n+1}},
$$
since we have $f_{*k} = (1)$ for each $k \in \{ 0, \, 2, \, \dots , \, 2n \}$. 
 
We consider next the case $d=-1$. As before, we have $n+1$ nontrivial induced maps, with $f_{*k}= (1)$ if $k/2$ is even and $f_{*k}= (-1)$ if $k/2$ is odd.

For $n$ even computing the Lefschetz zeta function for $f$ using equation \eqref{zeta_product}, we get
$$
\mathcal{Z}_f(t) = \prod_{k=0}^{2n} \det \, (I_{n_k} - t\, f_{*k})^{(-1)^{k+1}}= \frac{1}{(1-t)^{(\frac{n}{2} + 1)}} \frac{1}{(1+t)^{\frac{n}{2}}} = \frac{1}{(1-t^2)^{\frac{n}{2}}} \frac{1}{(1-t)},
$$
and for $n$ odd 
we have
$$
\mathcal{Z}_f(t) = \prod_{k=0}^{2n} \det \, (I_{n_k} - t\, f_{*k})^{(-1)^{k+1}}= \frac{1}{(1-t)^{\frac{n+1}{2}}} \frac{1}{(1+t)^{\frac{n+1}{2}}} = \frac{1}{(1-t^2)^{\frac{n+1}{2}}}.
$$
 
It is clear that $f$ satisfies the hypotheses of Theorem \ref{th_franks}. 
Then the results follow directly applying Theorem \ref{th_franks} to each of the expressions obtained for the Lefschetz zeta function $\mathcal{Z}_f(t)$. For $\mathcal{Z}_f(t) = \frac{1}{(1-t)^{n+1}}$ and $\mathcal{Z}_f(t) = \frac{1}{(1-t^2)^{\frac{n}{2}}} \frac{1}{(1-t)}$ we have $\text{MPer}_L(f) = \{1\}$, since any other expression of the same Lefschetz zeta function as a product of terms the form $(1\pm t^p)^{\pm 1}$ would provide at least the period $1$.

For  $\mathcal{Z}_f(t) =  \frac{1}{(1-t)^{\frac{n+1}{2}}} \frac{1}{(1+t)^{\frac{n+1}{2}}} = \frac{1}{(1-t^2)^{\frac{n+1}{2}}}$ it is clear that $\text{MPer}_L(f) = \emptyset$,  but taking into account the considerations of Remarks \ref{tres_maneres} and \ref{remark_nparell}, one has $\{1, \, 2\} \cap \text{Per}(f) \neq \emptyset$.  
\end{proof}

\begin{remark}
 A map $f$ is called \emph{Lefschetz periodic point free} if $L(f^m) = 0$, for all $m \in \N$. In \cite{graff-justyna} is claimed that there are no Lefschetz periodic point free maps on $\CPn$ and $\HPn$, that is, all self maps of $\CPn$ and $\HPn$ have periodic points. Here we see that in particular Morse--Smale diffeomorphisms in $\CPn$ and $\HPn$ always have fixed points, unless when $n$ is odd and $d=-1$, where in this case there are always fixed points or periodic points of period $2$.
\end{remark}

\section*{Acknowledgements}

This work is supported by the Ministerio de Ciencia, Innovaci\'on y Universidades, Agencia Estatal de Investigaci\'on grant PID2019-104658GB-I00 (FEDER), the Ag\`encia de Gesti\'o d'Ajuts Universitaris i de Recerca grant 2017SGR1617, and the H2020 European Research Council grant MSCA-RISE-2017-777911.

\end{document}